\documentclass[11pt,leqno]{article}

\usepackage{amssymb,amsmath,amsthm, url,rotating}
 \usepackage{graphicx,epsfig}
 \usepackage{xcolor,graphicx}
  \usepackage{hyperref}
\newcommand{\vp}{\varepsilon}
\newcommand{\bb}[1]{\mathbb{#1}}
\newcommand{\cl}[1]{\mathcal{#1}}

\newcommand{\ovl}{\overline}

\theoremstyle{plain}
\newtheorem{thm}{Theorem}[section]

\newtheorem{pro}[thm]{Proposition}

\theoremstyle{definition}

\newtheorem{dfn}[thm]{Definition}

\theoremstyle{remark}
\newtheorem{rem}[thm]{Remark}

\numberwithin{equation}{section}

\def\tilde{\widetilde}

\renewcommand{\tilde}{\widetilde}

\def\NN{\bb N}

\def\Z{\bb Z}

\def\NN{\bb N}

\begin{document}

\title{Quantum expanders and growth of
 group representations}

\author{by\\
Gilles Pisier\footnote{Partially supported   by ANR-2011-BS01-008-01.}\\
Texas A\&M University\\
College Station, TX 77843, U. S. A.\\
and\\
Universit\'e Paris VI\\
Inst. Math. Jussieu, \'Equipe d'Analyse Fonctionnelle, Case 186, \\
75252
Paris Cedex 05, France}

 \maketitle

\begin{abstract}
Let $\pi$ be a finite dimensional unitary representation of a group $G$ with
a generating symmetric $n$-element set $S\subset G$. Fix $\vp>0$.
Assume that the spectrum of $|S|^{-1}\sum_{s\in S} \pi(s) \otimes \overline{\pi(s)}$
is included in $ [-1, 1-\vp]$ (so there is a spectral gap $\ge \vp$).
Let $r'_N(\pi)$ be the number of distinct irreducible representations
of dimension $\le N$ that  appear in $\pi$. Then let $R_{n,\vp}'(N)=\sup r'_N(\pi)$ where the supremum runs over all $\pi$ with ${n,\vp}$ fixed.
We prove that there are positive constants $\delta_\vp$ and $c_\vp$
such that,
for all sufficiently large integer $n$ (i.e. $n\ge n_0$   with $n_0$ depending
on   $\vp$) and for all  $N\ge 1$, we have $\exp{\delta_\vp nN^2} \le R'_{n,\vp}(N)\le \exp{c_\vp nN^2}$. The same bounds hold if, in $r'_N(\pi)$, we 
count only the number of distinct irreducible representations
of dimension exactly $= N$.
\end{abstract}
\section{Introduction}
  We wish to formulate and answer a natural
extension of a question raised explicitly
by Wigderson in several lectures (see e.g. \cite[p.59]{Wi}) and also implicitly in \cite{MW}.  
Although the variant that we answer seems to be much easier,   it may shed some light
on the original question. Wigderson's question concerns the growth of the number $r_N(G)$ of distinct irreducible representations
of dimension $\le N$ that may appear
on a finite group $G$ when the order of $G$ is arbitrarily large and all that one knows is that $G$ admits
a generating set $S$ of $n$ elements for which the Cayley graph forms an
expander with a fixed spectral gap $\vp>0$. The problem is to find
the best bound of the form $r_N(G)\le R(N) $ with $R(N)$ independent
of the order of $G$ (but depending on $n,\vp$). 
We consider a more general framework: the finite group $G$ is replaced by a finite dimensional representation $\pi$ (playing the role of the regular representation $\lambda_G$
for finite groups) such that the representation $\pi\otimes \bar \pi$ admits a spectral gap, meaning that the trivial representation is isolated with a gap
$\ge \vp$ from the other irreducible components of $\pi\otimes \bar \pi$. 
When $\pi=\lambda_G$ we recover the previous notion of spectral gap.
Let then $r'_N(\pi)$ be the number of distinct irreducible representations
of dimension $\le N$ appearing in $\pi$ (note that $r_N(G)=r'_N(\lambda_G)$), and let $R'(N)$ denote the least upper bound $r'_N(\pi)\le R'(N)$ when the only restriction on $\pi$ is
that $n,\vp$ remain fixed (but the dimension of $\pi$ is arbitrary).
We observe that
the previously known bound for $R(N)$ namely $R(N)=e^{O(nN^2)}$
is also valid for  $R'(N) $ and also that $R(N)\le R'(N) $. 
Our main result, which follows from
the metric entropy estimate for quantum expanders in \cite{P}, is that this bound for $R'(N)$ is sharp:  there is
$\delta>0$ such that 
for all $n$   large enough (i.e. $\forall n\ge n_0(\vp)$) we have  $R'(N) \ge e^{\delta nN^2}$ for all $N$.

The term ``quantum expander" was coined in \cite{Ha,BA1,BA2} to which we refer for background  (see also \cite{Har,HHar}).

\section{Main result}
 Let $G$ be any group with a  finite generating set $S\subset G$
with $|S|=n$. For any unitary representation
$\pi:\ G \to H_\pi$ we set
$$\lambda(\pi ,S)=n^{-1}\sup\{ \Re\ \langle\sum_{s\in S} \pi(s)\xi,\xi\rangle\mid \xi\in 
{H_{\pi}^{{\rm inv}}}^\perp, \|\xi\|_{H_\pi}=1\}.$$
 where
$ {H_\pi^{{\rm inv}}}\subset H_\pi$ denotes the subspace of all $\pi$-invariant vectors. \\
When $S$ is symmetric, $\sum_{s\in S} \pi(s)$ being selfadjoint, the real part sign $\Re$ can be omitted.\\
We then set
$$\vp(\pi ,S)=1- \lambda(\pi ,S). $$
It will be useful to record here the elementary observation that
if $\pi$ is unitarily equivalent to the direct sum $\oplus_{i\in I} \pi_i$
of a family of unitary representations, then $\lambda(\pi ,S)=\sup_{i\in I}\lambda(\pi_i ,S)$ and hence
\begin{equation}\label{bd0}\vp(\pi ,S) =\inf_{i\in I}\vp(\pi_i ,S).\end{equation}
In particular, if $\pi_1$ is contained in $\pi_2$, then $\vp(\pi_1 ,S)\ge \vp(\pi_2 ,S)$.

We denote
$$\vp(G,S)=\inf\{\vp(\pi ,S)\}
$$
where
the infimum runs over all unitary representations
$\pi:\ G \to H_\pi$. 
Thus the condition $$\vp(G,S)>0$$ means that $G$
has Kazhdan's ``property (T)", (or in otherwords  is a Kazhdan-group), see \cite{BHV} for more background.

We start by the following result somewhat implicitly 
due to S. Wassermann \cite{Was} and explicitly proved in detail in \cite{HRV}.

\begin{pro}[\cite{Was,HRV}] \label{p1} For any $\vp>0$ there is a constant $c_\vp$
such that for any $n$, any group   $G$ and any $S\subset G$ with $|S|=n$ 
such that $\vp(G,S)\ge \vp$, 
the number $r_N(G)$ of distinct irreducible unitary representations
$\sigma:\ G\to B(H_\sigma)$ with $\dim(H_\sigma)\le N$   is majorized as follows:
\begin{equation}\label{bd1}
r_N(G)\le \exp{(c_\vp nN^2)}.\end{equation}
Of course, here distinct means up to unitary equivalence.
\end{pro}
\begin{rem} Note that it suffices
to prove a bound of the same form for
the number of distinct irreducible unitary representations
$\sigma:\ G\to B(H_\sigma)$ with $\dim(H_\sigma)= N$.
Indeed, if the latter number is denoted by $s_N(G)$,
we have
$r_N(G)=\sum_{d=1}^N  s_d(G)$, so that it suffices to have
a bound of the form  $s_d(G)\le \exp{(c'_\vp nd^2)}$ to 
obtain \eqref{bd1}. \\
See \cite{KN,LL} for some examples of estimates of the growth of $r_N (G)$. 
\end{rem}
We note that it was originally proved by Wang \cite{W} that
for any Kazhdan-group $G$ this number $r_N(G)$ is finite for any $N$.
There is an indication of proof of \eqref{bd1}  in \cite{Was}, and detailed proofs appear in \cite{HRV} (see also \cite{MW}).
We will prove a simple extension of this bound below.

Recall that a sequence $(G_k,S_k)$ of finite groups
equipped with generating sets $S_k\subset G_k$ such that
 $$\sup_k|S_k|< \infty,\quad |G_k|\to \infty \quad{\rm and}\quad \inf_{k}  \vp(G_k,S_k)   >0$$
is called an expander or an expanding family. This corresponds to the usual notion
 among {\it Cayley} graphs to which we restrict the entire discussion.\\
 Let $\hat G$ denote as usual the (finite) set of all irreducible unitary representations of a finite group $G$
 (up to unitary equivalence).\\
 We note in passing that
 it is well known
 (and this   also can be derived from  Proposition \ref{p1}) that any   expander satisfies
 \begin{equation}\label{bd3pre}
 \lim_{k\to \infty} \max\{\dim(H_\sigma)\mid \sigma\in \hat {G_k} \} = \infty.
 \end{equation}
 
  We refer the reader to the surveys \cite{HLW,Lu2} for more information on expanders.
 
The question raised by Wigderson in this context can be formulated as follows:

Let
$$R_{n,\vp}(N)=\sup\{ r_N(G)\}$$
where the supremum runs over all finite groups $G$ admitting a subset $S$
with  $ |S|=n$ such that $\vp(G,S)\ge \vp$. Actually the question is just as interesting for
arbitrary (Kazhdan) groups $G$, but it is more natural to restrict  to finite groups, because
there are infinite Kazhdan groups without {\it any} (nontrivial) finite dimensional representations.

Moreover, since, for a finite group $G$,   all   representations are weakly contained in the 
left regular representation $\lambda_G$, we have clearly by \eqref{bd0}
\begin{equation}\label{bd2} \vp(G,S)=\vp(\lambda_G,S).\end{equation}
 
 By \eqref{bd1}, we have
\begin{equation}\label{bd3}   R_{n,\vp}(N)  \le \exp{(c_\vp nN^2)}.\end{equation}
and a fortiori simply 
$ R_{n,\vp}(N) = \exp {O(N^2)}.$\\
Wigderson asked whether this upper bound can be improved.
More explicitly, what is the precise order of growth
of $\log R_{n,\vp}(N) $ when $N \to \infty$. Does it grow   like
$N$    rather than  like $N^2$ ? \\
The motivation for this question can be summarized like this: In \cite[Th. 1.4]{MW} an exponential
bound  $\exp {O(N)} $ is proved for a special class of groups $G$
(namely monomial groups), admitting a fixed spectral gap
with generating sets of very slowly growing size (but not bounded)
and it is asked whether the same exponential bound holds in general for $ R_{n,\vp}(N) $.
Moreover, in a remark following the proof of \cite[Th. 1.4]{MW},
Meshulam and Wigderson observe that
 for     any prime number $p>2$, there is a group
$G_p$ with a generating set of (unbounded) size $\log p$ 
admitting a fixed spectral gap and such that $r_p(G)\approx 2^p/p$.

 \begin{rem} By classical results, originating in the works of Kazhdan and Margulis (see e.g. \cite{Lu} or \cite[Cor. 2.4]{Lu2}), 
 for any fixed $m\ge 3$,  the family
 $\{SL_m(\Z_p)\mid p \ {\rm prime} \}$  is an expander, so that
 we have (for suitable ${\ell,\delta}$)
 $$R_{\ell,\delta}(N) \ge \sup\nolimits_p  r_N(   SL_m(\Z_p) ).$$
 Similarly, let ${\cl G}_k$ denote the symmetric group of all permutations of a $k$ element set.
  Kassabov \cite{Kas} proved that the family $\{  {\cl G}_k\mid k\ge 1\}$ forms
 an expanding family with respect to subsets $S_k\subset {\cl G}_k$ of a fixed size $\ell$
 and a fixed spectral gap $\delta>0$.
 Thus we find a lower bound
 $$R_{\ell,\delta}(N) \ge \sup\nolimits_k  r_N(   {\cl G}_k ).$$
 Quite remarkably,  it is proved in \cite{KLN} that the   family itself of {\it all}
 non-commutative finite simple groups forms an expander (for some suitable $n,\vp$).
 \end{rem}
  \begin{rem}
However, it seems  the resulting  lower bounds
 are still far from being exponential in $N$.
 Actually, in many important cases (see {\it e.g.} \cite{BG}), the proof that certain finite groups $G$
 give rise to  expanders uses the fact that  the smallest dimension
 of a (non-trivial) irreducible representation on $G$ is $\ge c|G|^a$
 for some $a>0$. 
  Then since $|G|=\sum_{\pi\in \hat G} \dim(\pi)^2$ 
   the cardinal of $\hat G$  is bounded above by $|G|^{1-2a}/c^2$. Therefore, for any $N\ge c|G|^a $   we have $r_N(G)\le |G|^{1-2a}/c^2 \le c' N^{(1/a)-2}$, so that the resulting growth implied for $R_{n,\vp}(N)$ is at most polynomial in $N$.
 (I am grateful to N. Ozawa for drawing my attention to this point).
 \end{rem}
 Nevertheless, we have:
 
  \begin{rem}\label{comkas} (Communicated by Martin Kassabov).
 For suitable ${n,\vp}$ the numbers $R_{n,\vp}(N)$
 grow faster than any power of $N$. In fact, we will prove the\\
{\bf Claim\ :}  There
 is an expanding family of Cayley graphs $(G_k)$ of groups 
 generated by 3 elements with a positive spectral gap $\vp$
 and such that 
 for $N_k=2^{3k}-2$, 
 $G_k$ admits $2^{k^2}$ distinct irreducible representations
 of dimension $ N_k$. \\
 From this claim
 follows that $R_{3,\vp}(N_k)\ge 2^{k^2}\ge 2^{(\log(N_k))^2}$,
 say for all $k$ large enough, and hence
 $$R_{n,\vp}(N)\ge 2^{(\log(N))^2} \ \text{for infinitely many $N$'s.}$$  
 To prove the claim we use the ideas from \cite{Kas1}.
 Let $\cl R_k$ denote the (finite) ring $M_k(F_2)$
 of $k\times k$ matrices with entries in the field with 2 elements.\\
 It is known  
 that the cartesian product $ \Pi_k= {\cl R}_k^{2^{k^2}}$ of $|\cl R_k|=2^{k^2}$ copies
 of $\cl R_k$ is generated by 3 elements. 
 Indeed, $\cl R_k$ itself is generated as a ring
 by two elements, e.g. $a=e_{12}$
 and the shift $b=e_{12}+e_{23}+\cdots+e_{k-1 k}+e_{k1}$,
 then $ \Pi_k$ is generated as a ring by $\{A,B,C\}$
 where $A$ (resp. $B$) is the element with all coordinates equal to $a$
 (resp. $b$), and $C$ is such that its coordinates are in
 one to one correspondence with the elements of $\cl R_k$.
 To check this,
  let $R\subset \Pi_k$ be the ring generated
 by $\{A,B,C\}$. Note, by the choice of $C$, the following easy observation:
 for any coordinate $i$, there is $x\in R$
 such that $x_i=0$ but $x_j\not=0$ for all $j\not=i$.
 For any subset $I$ of the index set
 let $p_I:\ R\to {\cl R}_k^I$ be the coordinate projection.
One can then prove by induction on $m=|I|$ that
$p_I(R)={\cl R}_k^I$ for all $I$. Indeed, assume 
the fact established for $m-1$.
For any $I$ with $|I|=m$  we pick $i\in I$ and we consider
the set ${\cl I}=\{ y\in {\cl R}_k^{I\setminus i}\mid  (0,y)\in p_I(R)\}$.
By the induction hypothesis,  ${\cl I}$ is an ideal in
${\cl R}_k^{I\setminus i}$, but,   
  since ${\cl R}_k$ is simple, the above observation
  implies that ${\cl I}=  {\cl R}_k^{I\setminus i}$,
  and since $a,b$ generate ${\cl R}_k$ we have $p_{\{i\}}(R)={\cl R}_k $, so we obtain
    $p_I(R)={\cl R}_k^I$.
 
This implies that
 the free associative ring $\Z\langle x,y,z\rangle$
 (in 3 non-commutative variables) 
 can be mapped onto the  product $\Pi_k$.
 Consider now the group $EL_3(\Z\langle x,y,z\rangle)$
 generated by the elementary matrices in $GL_3(\Z\langle x,y,z\rangle)$.
 This is a noncommutative universal lattice
 in the terminology of \cite{Kas1,EJZ}.
 First observe that $EL_3(\Z\langle x,y,z\rangle)$
 is generated by $3$ elements. Indeed, let $\alpha,\beta$
 generate $SL_3(\Z)$. Then $\alpha,\beta,\gamma$
 will generate $EL_3(\Z\langle x,y,z\rangle)$
 where $\gamma=\left(\begin{matrix} 1\ x\ y\\
 0\ 1\ z\\ 0\ 0\ 1
 \end{matrix}\right)$. Moreover, by \cite[Th.1.1]{EJZ}
 $EL_3(\Z\langle x,y,z\rangle)$
 has Kazhdan's property T.
 It follows that the groups
 $$G_k=EL_3(\Pi_k) $$
 have expanding generating sets with 3 elements.
 But it turns out that $G_k$  can be identified with  the product 
 $$SL_{3k}(F_2) ^{2^{k^2}}.$$
 Indeed, firstly one
 easily   checks the natural isomorphism
 $EL_3({\cl R}_k^{2^{k^2}})\simeq  EL_3({\cl R}_k)^{2^{k^2}}$,
 secondly it is well known that, since $F_2$ is a field,
 $EL_n(F_2)=SL_n(F_2)$ for any $n$, 
 and hence (taking $n=3k$) we have a natural isomorphism
  $EL_3({\cl R}_k)= SL_{3k}(F_2)$; this
  yields the identification $G_k=SL_{3k}(F_2) ^{2^{k^2}}.$
  
 To conclude, we will use the fact that $SL_{3k}(F_2)$
 admits a nontrivial irreducible representation  $\pi$ with dimension $N_k=2^{3k}-2$.
 (Just consider its action by permutation on the projective space,
 which has $2^{3k}-1$ elements; the action
   is transitive and doubly transitive, therefore
   the associated Koopman representation
  $\pi$ is irreducible and of dimension $2^{3k}-2$).
 This immediately produces ${2^{k^2}}$ distinct
    irreducible representations of dimension  $N_k$
 on $SL_{3k}(F_2) ^{2^{k^2}}.$ Indeed, it is an elementary fact
 that if $\Gamma=\Gamma_1\times\cdots\times \Gamma_m$ is a product group, and if
 $\pi_1,\cdots \pi_m$ are arbitrary
 nontrivial irreducible representations
 on the factor groups $\Gamma_1,\cdots, \Gamma_m$, then
 the representations $\tilde \pi_j$ defined on $\Gamma$
 by $\tilde \pi_j(g)= \pi_j(g_j)$ are distinct 
 (meaning not unitarily equivalent),  irreducible  on $\Gamma$
 and $\dim(\tilde \pi_j)=\dim(  \pi_j)$ for any $j$.
 So taking all $\Gamma_j$'s equal to $SL_{3k}(F_2)$, with $\pi_j=\pi$
 and $m={2^{k^2}}$, we obtain the announced claim.
 \end{rem}

In any case, the problem of finding the correct behaviour
 of $  \log R_{n,\vp}(N) $ (or of $R_{n,\vp}(N) $ itself) when $N \to \infty$ appears to be still
 wide open.

   In this paper we consider a modified version of this question involving 
``quantum expanders" and  show that for this (much easier) modified version, $N^2$ is the correct order of growth.

The term  ``quantum expander" was introduced in \cite{Ha}  and     \cite{BA1,BA2}, independently, to designate
a sort of non-commutative, or matricial, analogue of expanders, as follows.

Fix an integer $n$. 
Consider an
$n$ tuple of $N \times N$ unitary matrices, say  $u=(u_j)\in U(N)^n$.
We view each of  them $u_j$ as a linear operator on the $N$-dimensional Hilbert space $H$. Then $u_j\otimes \ovl{u_j}$
is naturally viewed as a linear operator on
 the (Hilbert space sense)  tensor product $H\otimes \bar H$.
Using the (canonical) identification $H^*\simeq \bar H$, the tensor product
$H\otimes \bar H$ can be isometrically identified with the space
of linear operators from $H$ to $H$ equipped with the Hilbert-Schmidt norm
denoted by $\|\ \|_2$ (sometimes called the Frobenius norm in the present finite dimensional context). Then, the identity operator $Id_H:\ H\to H$
defines a distinguished element of $H\otimes \bar H$ that we   
 denote   by $I$.

\def\tr{{\rm tr}}
We set
$$\lambda(u)=n^{-1}\sup\{\Re \langle (\sum_1^n  u_j \otimes \bar u_j )\xi,\xi\rangle \mid \xi\in H\otimes \bar H ,\  \xi \perp  I, \ \|\xi\|_{H\otimes \bar H}=1      \},   $$
and
$$\vp(u)=1-\lambda(u).$$
In other words,  with the preceding identifications, the condition $\vp(u)\ge \vp$ means that
for any $x\in M_N$ with $\tr(x)=0$  we have
$$ \Re \sum \tr (u_j xu_j^* x^*)  \le (1-\vp) \|x\|_2,$$
where $\|x\|_2= (\tr(x^*x))^{1/2}$.\\
When $T=\sum_1^n  u_j \otimes \bar u_j$ is self adjoint
(in particular when the set $\{u_1,\cdots ,u_n\}$ is selfadjoint)
the real part $\Re$ can be omitted in the two preceding lines.

In group theoretic language, if $\pi:\ {\bf F}_n \to U(N)$
is the group representation on the free group ${\bf F}_n$, equipped with
a set of $n$ free generators
$S=\{g_1,\cdots,g_n\}$, such
that  $\pi(g_j)=u_j$ ($1\le j\le n$), then we have
$$\vp(u)=\vp(\pi\otimes \ovl{\pi},S).$$

\begin{dfn}  A sequence $\{ u(k)\mid k\in \NN\}$ 
with each $u(k)\in U(N_k)^n$ such that $N_k\to \infty$
(with $n$ remaining fixed) and $\inf_{k}  \{\vp( u(k)) \} >0$
is called a quantum expander.
We say that   $n$ is  its degree and
 $\inf_{k}  \{\vp( u(k)) \} >0$ its  spectral gap.\end{dfn} 

\begin{rem}
The existence of quantum expanders can be deduced as follows from
that of expanders. Recalling \eqref{bd2}, assume given
a finite group $G$ and  $S\subset G$ as before such that $\vp(G,S)=\vp(\lambda_G,S)\ge \vp>0$.
Recall that each $\sigma\in \hat G$ is contained in  $\lambda_G$.
Let $\pi\in \hat G$. Since any representation on $G$ without invariant vectors, being
  a direct sum of non trivial irreps, is weakly contained in $\lambda_G$, the representation
  $\rho=\pi\otimes \ovl{\pi}$ restricted to ${H_\rho^{{\rm inv}}}^\perp$ is weakly contained in
  the non trivial part of  $\lambda_G$.
  In particular, we have by \eqref{bd0}
  $$\lambda(\rho,S)  \le \lambda(\lambda_G,S)   . $$
  Therefore, we have
  $$\vp( \pi\otimes \ovl{\pi},S)\ge \vp( \lambda_G,S)\ge \vp.$$
   Thus if we are given an expander   $(G_k,S_k)$ as above, 
   say with $S_k=\{s_1(k),\cdots,s_n(k)\}$,  we can choose
   by \eqref{bd3pre} $\sigma_k\in \hat G_k$ such that
   $\dim(H_{\sigma_k})\to \infty$, and  
   if we set $u_j(k)= \sigma_k(s_j(k))$    ($1\le j\le n$),
   then $u(k)=\{u_1(k),\cdots,u_n(k)\}$ forms a quantum expander.
\end{rem}

The next statement is a simple generalization 
of Proposition \ref{p1}
\begin{pro} \label{p2}
For any $0<\vp<1$ there is a constant $c'_\vp>0$ for which the following holds.
Let $G$ be any group and let $\pi:\ G \to B(H)$ be any unitary representation on a finite dimensional Hilbert space $H$.
Let us assume that there is an $n$-element subset $S\subset G$
and $\vp>0$ such that
$$\vp(\pi \otimes \ovl{\pi},S)\ge \vp .$$ In other words,
 $\pi$ satisfies the following spectral gap condition:
\begin{equation}\label{eq1}\lambda( \pi\otimes \ovl{\pi},S)
\le  1-\vp\end{equation}
Let $\pi=\oplus_{t\in T} \pi_t$
be the decomposition into {\it distinct} irreducibles (where each $\pi_t$ has multiplicity $d_t\ge 1$),
then
\begin{equation}\label{bd4}|\{t\in T\mid \dim (\pi_t)\le N\}|\le \exp{c'_\vp nN^2}.\end{equation}
\end{pro} 
  \begin{proof}
  Let $\sigma=\oplus_{t\in T} \pi_t$ be the direct sum where each component is included with multiplicity equal to $1$.
  We may clearly view $\sigma$ as a subpresentation of $\pi$, acting on a subspace $K\subset H$
  so that the orthogonal projection $Q:\ H\to K$ is intertwining, i.e. satisfies
  $Q\pi=\sigma Q$. Then we also have $(Q\otimes \bar Q)(\pi\otimes \bar \pi)=(\sigma\otimes \bar \sigma)(Q\otimes \bar Q)$,
  from which it is easy to derive that if we denote   $V_\pi=H_{\pi\otimes \bar \pi}^{{\rm inv}}$,
  we have $(Q\otimes \bar Q) V_\pi=V_\sigma$ and $(Q\otimes \bar Q) V^\perp_\pi=V^\perp_\sigma$.
  This implies 
$$\lambda( \sigma\otimes \ovl{\sigma},S)\le \lambda( \pi\otimes \ovl{\pi},S)\le 1-\vp
.$$
Thus,  replacing $\pi$ by $\sigma$, we may  as well assume that the multiplicities $d_t$   are all equal to $1$.

  Let $H=\oplus_{t\in T} H_t $ denote the decomposition corresponding to $\pi=\oplus_{t\in T} \pi_t $.
   We have $\pi\otimes \bar \pi=\oplus_{t,r\in T} \pi_t \otimes\ovl{ \pi_r}$,
    with associated decomposition $H\otimes \bar H=\oplus_{t,r\in T}  H_t\otimes \ovl{H_r}$. From this 
    follows that the subspace $V_\pi\subset H\otimes \bar H$  of $\pi\otimes \bar \pi$-invariant vectors is equal to
    $\oplus_{t,r\in T}  V_{t,r}$ where $V_{t,r}\subset H_t\otimes \ovl{H_r}$ is the  subspace
    of invariant vectors of $   \pi_t \otimes\ovl{ \pi_r}$. Since for any $t\not= r\in T$,  $\pi_t\not \simeq \pi_r$, by Schur's lemma     
    $V_{t,r}=\{0\}$, and hence $V_\pi\subset \oplus_{t\in T} V_{t,t}$.  In particular,
    this shows
    that   $$\forall  t\not= r\in T \quad H_t\otimes \ovl{H_r}\subset V_\pi^\perp.$$   
   Let
   $T'=\{t\in T\mid \dim (\pi_t)=N\}.$
   It suffices to show an estimate of the form
   \begin{equation}\label{eqbd5}|T'|\le \exp{c_\vp nN^2}.\end{equation}
   Let $\cl H$ be the Hilbert space obtained by equipping $M_N^n$ with the norm
   $$\|x\|^2_{\cl H}=   N^{-1} n^{-1}\sum_1^n{\rm tr} (x_j^*x_j)   .$$  
   Let $S=\{s_1,\cdots,s_n\}$.
   For any $t\in T' $ we define $x(t)\in M_N^n$ by
   $$x(t)_j= \pi_t(s_j) \quad 1\le j\le n.$$
   Note that, by our normalization,  $\|x(t)\|_{\cl H}=1$ for any $t\in T'$. Moreover,
   since for any $t\not= r\in T$  $\pi_t\not \simeq \pi_r$, by Schur's lemma the representation 
    $\pi_t\otimes \ovl{ \pi_r}$ has no invariant vector, and hence lies inside $(\pi \otimes \ovl{ \pi})_{|V_\pi^\perp}$.
 Therefore, by \eqref{bd0}
      $$\lambda( \pi_t\otimes \ovl{ \pi_r},S)\le 
      \lambda( \pi\otimes \ovl{ \pi},S),
      $$
      and hence for any unit vector $\xi \in H_{\pi_t} \otimes \ovl{H_{\pi_r}}$ we have
      $$n^{-1} \Re ( \sum_{s\in S} (\pi_t \otimes \ovl{\pi_r}) \xi ,\xi \rangle)\le 1-\vp .  $$
    In particular, if $t\not= r\in T'$, we may realize $\pi_t,{ \pi_r}$ as representations on the same $N$-dimensional space, so that taking $\xi=N^{-1/2} I $ we find
    $$ \Re \langle  x(t)   ,x(r)     \rangle_{\cl H} 
      = (nN)^{-1}  \Re\left( \sum_{s\in S}{\rm tr} (\pi_t(s)^*\pi_r(s)) \right) \le  1-\vp,$$
    which implies
    $$\|x(t) -x(r)\|_{\cl H}\ge \sqrt{2\vp}.$$
    Thus we have $|T'|$ points in the unit sphere of ${\cl H}$ that are $\sqrt{2\vp}$-separated. Since $\dim({\cl H})=nN^2$,
    \eqref{eqbd5} follows immediately by a well known elementary volume argument (see e.g.
    \cite[p. 57]{Pvol}).
    \end{proof}
     \begin{rem} To derive Proposition \ref{p1} from the preceding statement,
     consider, in the situation of Proposition \ref{p1}, a finite set $\{\sigma_t\mid t\in T\}$ of distinct finite dimensional 
     irreducible representations of $G$, 
       let $\pi $ be their direct sum
     and let $\rho=\pi\otimes \ovl{\pi}$.
     By the assumption in  Proposition \ref{p1}, we know $\vp(\rho,S)\ge \vp$, and hence \eqref{bd4}
      implies $|T|\le  \exp{c'_\vp nN^2}$. Applying this to $\pi=\lambda_G$,
     this shows that Proposition
     \ref{p2} contains Proposition
     \ref{p1}.
    \end{rem}
    
    For any  finite dimensional unitary representation $\pi:\ G \to B(H)$ on an arbitrary group,     let us denote by $r'_N(\pi)$ the number of distinct irreducible representations appearing in the decomposition of $\pi$ of dimension at most $N$. Let then
    $$R'_{n,\vp}(N)=\sup  r'_N(\pi)$$
    where the sup runs over all $\pi$'s and $G$'s admitting an 
     $n$-element generating set $S\subset G$ such that
     $$\vp(\pi\otimes \bar \pi,S)\ge \vp.$$
     Note that $r'_N(\lambda_G)=r_N(G)$ and hence
     $$R_{n,\vp}(N)\le R'_{n,\vp}(N).$$
     With this notation \eqref{bd4}
     means that
     $$R'_{n,\vp}(N)\le \exp{c'_\vp nN^2}.$$
     
     While it seems very difficult to give a good lower bound for $R_{n,\vp}(N)$, 
     we can answer the analogous question for $R'_{n,\vp}(N)$: Indeed, 
     the main result of \cite{P} (see \cite[Th. 1.3]{P}),   which follows,  implies
     the desired lower bound when reformulated in terms of representations.
     
\begin{thm}[\cite{P}]\label{goal}  For each $0<\vp<1$, there is a 
  constant $ \beta_\vp>0$   such that    and
for all sufficiently large integer $n$ (i.e. $n\ge n_0$   with $n_0$ depending
on   $\vp$) and for all  $N\ge 1$,  there is a   subset 
${\cl T}\subset U(N)^n$  with 
$$|{\cl T}|\ge \exp{\beta_\vp  nN^2}$$
  such that  
$$ \forall u\not= v\in {\cl T}\quad  \|\sum\nolimits_1^n  u_j \otimes \ovl{ v_j}\|\le n(1-\vp)\quad ({\rm we\  call\  these\ } ``\vp-{\rm separated}" ) ,$$
and $\vp(u)\ge \vp$
for all $u\in {\cl T}$  
   (we  call  these  ``$\vp$-quantum \ expanders"). \\
More precisely, for all $u\in {\cl T}$ we have
$$\|(\sum u_j \otimes \ovl{u_j})_{|I^\perp} \|\le n(1-\vp).$$
\end{thm}
   \begin{thm}\label{p3} The estimate in Proposition \ref{p2} is best possible in the sense that for any $0<\vp<1$  there is
  a constant $\beta_\vp >0$ such that for any $n$ large enough (i.e. $n\ge n_0(\vp)$),  for any $N\ge 1$ there is a group $G$ and a finite dimensional representation $\pi$ on $G$ 
  satisfying \eqref{eq1} and 
  admitting a decomposition $\pi=\oplus_{t\in T} \pi_t$,  
     with  distinct  irreducibles $\pi_t$ each with multiplicity $1$ (or any specified value $\ge 1$) and acting on an $N$-dimensional space,  with
   $$|T|\ge \exp {\beta_\vp  nN^2}.$$
 \end{thm} 
   \begin{proof} Fix $N>1$.  Let $T\subset U(N)^n$ be the subset appearing in Theorem \ref{goal}, i.e. $T$ is such that 
   $|T|\ge \exp {\beta_\vp nN^2}$ and
   $\forall t\not= r\in T$ we have
   \begin{equation}\label{eq2}\|\sum t_j \otimes \bar r_j \|\le n(1-\vp),\end{equation}
   and also
  \begin{equation}\label{eq3}\|(\sum t_j \otimes \bar t_j )_{|I^\perp}\|\le n(1-\vp).\end{equation}
   Let $s_j=\oplus_{t\in T} t_j \in U(m)$ with $m=|T| N$, and let $G\subset  U(m)$ be the subgroup  generated by $S=\{s_1,\cdots,s_n\}$. Note that $\pi(G)\subset    \oplus_{t\in T} M_N$. Let $\pi:\ G \to U(m)$ be the  inclusion map viewed as a representation on $G$.
   Let $P_t: \oplus_{t\in T} M_N \to M_N$ be the $*$-homomorphism corresponding to the projection
   onto the coordinate of index $t$.
   For any $t\in T$, let $\pi_t:\ G\to U(N)$ be the representation defined
   by $\pi_t = P_t(\pi)$.
   Then, by definition, we have $\pi =\oplus_{t\in T} \pi_t $.
   By the spectral gap condition \eqref{eq3} the commutant of  $\pi_t(S)$ (which is but the commutant of $\{t_1,\cdots,t_n\}$)
   is reduced to the scalars, so $\pi_t$ is irreducible, and by \eqref{eq2} for any $t\not= r\in T$
   the representations $\pi_t$ and  $\pi_r$ are not unitarily equivalent.
  \end{proof}  
       \begin{rem}
       In particular, this means that
         $\forall n\ge n_0(\vp)$ and $\forall N$
 $$R'_{n,\vp} (N)\ge \exp {\beta_\vp nN^2}.$$
 
       \end{rem}
       
       {\it Acknowledgement} I am indebted to 
       Martin Kassabov
       for letting me include Remark \ref{comkas} and
        very grateful for his patient explanations on its contents.


\begin{thebibliography}{100}
         
         
           \bibitem{BHV}     B. Bekka, P. de la Harpe and A. Valette,   Kazhdan's property (T). 
         Cambridge University Press, Cambridge, 2008.
           
            \bibitem{BA1} A. Ben-Aroya and A. Ta-Shma, Quantum expanders and the quantum entropy difference problem, 2007, arXiv:quant-ph/0702129. 
            \bibitem{BA2} A. Ben-Aroya,  O. Schwartz, and A. Ta-Shma,   Quantum expanders: motivation and constructions. Theory Comput. 6 (2010), 47-79.
        \bibitem{BG} J. Bourgain and A. Gamburd, Uniform expansion bounds
for Cayley graphs of $SL_2(F_p)$, Ann.  Math., 167 (2008), 625--642.
        
        \bibitem{EJZ}
  M. Ershov and A. Jaikin-Zapirain,   Property (T) for noncommutative universal lattices, Invent. Math. 179 (2010), no. 2, 303--347. 
  
               \bibitem{HRV}  P. de la Harpe,  A.G. Robertson  and A.  Valette,    On the spectrum of the sum of generators for a finitely generated group. Israel J. Math. 81 (1993),  65Ð96.
               
              \bibitem{Har}  A. Harrow,   Quantum expanders from any classical Cayley graph expander. Quantum Inf. Comput. 8 (2008),  715--721. 
               
            \bibitem{HHar}   M. Hastings and A. Harrow,  Classical and quantum tensor product expanders. Quantum Inf. Comput. 9 (2009),   336--360. 
            
             \bibitem{Ha} M. Hastings, Random unitaries give quantum expanders, Phys. Rev. A (3) 76 (2007), no. 3, 032315, 11 pp.
                        

         
         
       
            
     \bibitem{HLW} S. Hoory, N. Linial, and A. Wigderson,
     Expander graphs and their applications.
     Bull.
Amer. Math. Soc.
 43   (2006),  439-561.
     
      \bibitem{Kas}   M. Kassabov,   Symmetric groups and expander graphs. Invent. Math. 170 (2007), no. 2, 327Ð354. 
      
                \bibitem{Kas1}   M. Kassabov,
      Universal lattices and unbounded rank expanders, 
       Invent. Math. 170 (2007), no. 2, 297--326.

      
         \bibitem{KLN} M. Kassabov, A. Lubotzky and N. Nikolov,  Finite simple groups as expanders. Proc. Natl. Acad. Sci. USA 103 (2006),  6116--6119. 
  
     \bibitem{KN}  M. Kassabov and  N. Nikolov,  Cartesian products as profinite completions. Int. Math. Res. Not. 2006, Art. ID 72947, 17 pp.
      
     \bibitem{LL} M. Larsen and A. Lubotzky,
Representation growth of linear groups,  J. Eur. Math. Soc. 10 (2008), 351--390. 
           \bibitem{Lu} A. Lubotzky.
{\it Discrete groups, expanding graphs and invariant measures.}
Progress in Math. 125.
Birkh\"auser, 1994.
        \bibitem{Lu2} A. Lubotzky. Expander graphs
in pure and applied mathematics, Bull. Amer. Math. Soc. 49 (2012) 113--162.
        
         \bibitem{MW}   R. Meshulam and A. Wigderson, Expanders in group algebras. Combinatorica 24 (2004), 659-680.  
 \bibitem{Pvol} G. Pisier, {\it The volume of Convex Bodies and Banach
Space Geometry }. (Book) Cambridge University Press.1989.
    \bibitem{P} G. Pisier, Quantum Expanders and Geometry of Operator Spaces, J. Eur. Math. Soc. (JEMS) 16 (2014),   1183--1219.

 \bibitem{W}
P. S.  Wang, On isolated points in the dual spaces of locally compact groups. Math. Ann. 218 (1975), 19-34. 
     \bibitem{Was} S. Wassermann,  $C^*$-algebras associated with groups with Kazhdan's property T. Ann. of Math.  134 (1991),  423-431.
  
   \bibitem{Wi} A. Wigderson, lecture notes for the 22nd
mcgill invitational workshop on
computational complexity,
Bellairs Institute
Holetown, Barbados
Lecturers:
Ben Green and
AviWigderson.
 
  
  
  
  \end{thebibliography}
 \end{document}